\documentclass{article}
\usepackage{graphicx}
\usepackage[english]{babel}
\usepackage{amsmath}
\usepackage{amsthm}
\usepackage{amscd}
\usepackage{pb-diagram}
\usepackage{algorithm}
\usepackage{algpseudocode}
\usepackage{amsmath,amsthm,amssymb}

\usepackage{comment}
\usepackage{amssymb}
\usepackage{subcaption}
\usepackage{color}
\usepackage{enumitem}
\usepackage{faktor}
\usepackage[normalem]{ulem}  

\usepackage{amsmath, amssymb, amsthm}  
\usepackage{mathtools}  
\usepackage{bm}  
\usepackage{bbm}  
\usepackage{physics}  

\usepackage{tikz}
\usetikzlibrary{arrows, automata, positioning}

\usepackage{graphicx}  
\usepackage{float}  

\usepackage{geometry}  
\usepackage{times}  
\usepackage{setspace}  

\usepackage{hyperref}
\hypersetup{
    colorlinks=true,
    linkcolor=blue,
    citecolor=blue,
    urlcolor=blue
}

\theoremstyle{definition}
\newtheorem{definition}{Definition}
\theoremstyle{plain}
\newtheorem{theorem}{Theorem}
\newtheorem{example}{Example}
\newtheorem{proposition}{Proposition}
\newtheorem{lemma}{Lemma}
\newtheorem{corollary}{Corollary}

\theoremstyle{remark}
\newtheorem{remark}{Remark}

\usepackage{authblk}

\usepackage{algorithm}
\usepackage{algorithmicx}
\usepackage{algpseudocode}

\newcommand{\mo}{\operatorname{mo}}
\newcommand{\cl}{\operatorname{cl}}

\newcommand{\img}{\operatorname{im}}

\newcommand{\esol}{\operatorname{eSol}}

\def\mathobj#1{\mbox{$#1$}}

\def\ZZ{\mathobj{\mathbb{Z}}}

\newcommand{\dom}{\operatorname{dom}}
\newcommand{\Inv}{\operatorname{Inv}}

\def\cT{\text{$\mathcal T$}}

\def\cV{\text{$\mathcal V$}}

\def\articletheorems{

}
\articletheorems

\renewenvironment{proof}{{\bf Proof:\ }}{\qedsymbol}

\title{Graph Multivector Persistence: A Unified Framework for Dynamic Systems} 
 \author{Donald Woukeng}
\affil{African Institute For Mathematical Science (AIMS)\\

Kigali, Rwanda.\\ 
donald.woukeng@aims.ac.rw}
\date{}

\begin{document}

\maketitle
\begin{abstract}
We introduce a persistence-type invariant for finite weighted graphs
based on combinatorial multivector dynamics. 
For each threshold parameter, a relation matrix determines a graph
multivector field, whose induced directed dynamics admits a
Morse decomposition given by its strongly connected components.
As the threshold varies, these multivector fields form a monotone
refinement family. 
We define the Morse persistence diagram by recording the birth and
death of Morse sets along this filtration.

The construction is purely combinatorial and does not rely on
simplicial homology or persistence modules. 
We prove that the resulting persistence diagram is stable with respect
to perturbations of the relation matrix in the sup norm.
Each Morse set furthermore carries a combinatorial Conley index,
yielding a topologically enriched invariant for multiscale graph
structure.
\end{abstract}\section{Introduction}

Graphs are among the most universal mathematical objects. 
They encode relations, transitions, interactions, and dependencies,
and they arise naturally in stochastic systems, network dynamics,
logical inference, data science, and quantum information.
Yet, while graph-theoretic descriptors such as connectivity,
centrality, or spectral invariants are well developed,
they do not fully capture how \emph{coherent recurrent structures}
emerge, organize, and persist across scales.

The aim of this work is to introduce a unified combinatorial framework
for detecting and quantifying such persistent organization in graphs.
Our approach extends the theory of combinatorial multivector fields
introduced by Mrozek~\cite{mrozek2017conley}, and further developed in
\cite{Woukeng_2024,cote2025data}, to the setting of weighted graphs,
and integrates it with a persistence-type viewpoint inspired by
topological data analysis \cite{cohen2005stability,chazal2016structure}.

The central object of the theory is the \emph{graph multivector field}.  
Given a weighted relation matrix $W$ on a finite graph,
we construct, for each threshold parameter $\lambda$,
a partition of the vertex--edge set into multivectors.
Each multivector groups together vertices and edges that are strongly
linked according to the thresholded relation.
As the threshold varies, these partitions evolve monotonically,
producing a filtration of combinatorial structures on the graph.

From each level of this filtration, we derive an associated directed
graph, the $M$-graph whose strongly connected components define the
\emph{Morse decomposition} at that scale.
These Morse sets represent recurrent or internally coherent components
of the graph relative to the multivector structure.
Tracking how these Morse sets appear and merge as $\lambda$ varies
leads to a persistence-type invariant:
the \emph{Morse persistence diagram}.
Each point $(\lambda_b,\lambda_d)$ records the birth and death thresholds
of a Morse set, that is, the parameter range over which a recurrent
structure persists.

Although the resulting diagram resembles a classical persistence
diagram, our setting is intrinsically combinatorial.
We do not construct persistence modules in the sense of
\cite{zomorodian2004computing}, nor rely on functorial homology of nested
simplicial complexes as in \cite{EdelsbrunnerHarer2010CT}.
Instead, persistence arises from refinement relations between
multivector partitions and from monotonicity properties of Morse
decompositions.
The bottleneck stability of the diagram follows directly from these
refinement properties, in the spirit of stability results in
persistent homology \cite{cohen2005stability,chazal2016structure}.
Thus stability is a structural consequence of the combinatorial
construction itself.

The present work sits at the intersection of several established
theories.
On the dynamical side, Morse decompositions and Conley index theory
provide a topological framework for understanding invariant and
recurrent structures \cite{conley1978isolated,mischaikow2002conley}.
Discrete and combinatorial analogues have been developed through
multivector fields and computational Conley theory
\cite{mrozek2017conley,kaczynski2004computational}.
On the topological side, persistent homology has become a central tool
for extracting multiscale features from data
\cite{EdelsbrunnerHarer2010CT}.
Applications to network and graph data include persistent analyses of
complex networks and brain connectivity
\cite{horak2009persistent,petri2014homological}.
The framework developed here differs in that persistence arises from
combinatorial dynamics on graphs rather than from nested simplicial
complexes, thereby linking recurrence theory with multiscale graph
analysis.

The construction is intentionally flexible.
By choosing different interpretations of the relation matrix $W$,
the same mechanism applies across diverse domains:

\begin{itemize}
  \item In \emph{Markov chains}, $W$ derives from transition
  probabilities, and Morse sets may correspond to recurrent or metastable
  classes, closely related to Perron cluster and metastability
  analysis \cite{schutte2003biomolecular}.
  
  \item In \emph{temporal or interaction networks}, $W$ measures
  aggregated interaction strength, and persistence quantifies the
  robustness of connectivity patterns.
  
  \item In \emph{logical or causal systems}, $W$ represents implication
  or correlation strengths, revealing stable feedback or inference
  cycles.
  
  \item In \emph{quantum systems}, $W$ may encode correlation or
  entanglement strength between subsystems, identifying groups that
  remain coherently coupled.
\end{itemize}

Moreover, the theory admits a topological enrichment.
Each Morse set carries a combinatorial Conley index,
whose Betti numbers describe its internal structure.
This yields a Conley--enriched Morse persistence diagram,
providing not only lifespan information but also topological signatures
of recurrent components.

Beyond theoretical analysis, the persistence diagram can be used as a
feature representation for data-driven tasks.
Vectorizations such as the algebraic coordinates introduced by
Adcock, Carlsson, and Carlsson \cite{adcock2013ring}
transform persistence diagrams into finite-dimensional feature vectors,
while alternative summaries such as persistence landscapes
\cite{bubenik2015statistical} provide stable functional embeddings.
In this way, the framework connects combinatorial dynamics,
topological structure, and statistical learning.

The remainder of the paper develops the theory systematically.
We first introduce graph multivector fields and establish their
monotonicity and refinement properties.
We then construct the associated Morse decompositions and define the
Morse persistence diagram.
Stability with respect to perturbations of the relation matrix is
proved using purely combinatorial arguments.
Finally, we discuss topological enrichment via Conley indices and
outline applications to dynamical systems, networks, and data analysis.
\section{Graph Multivector Fields}

Multivector fields were originally introduced by Mrozek \cite{lipinski2019conley} in the context of
combinatorial dynamics as a tool to capture invariant sets and Morse
decompositions of dynamical systems \cite{Woukeng_2024}. The key idea is to replace the
continuous notion of a vector field by a partition of the underlying
combinatorial space into so--called multivectors, and then to analyze
dynamics through this discrete structure, that is suppose to mimic the behavior of the continuous system.

In this work we take this idea and specialize it to graphs. Our basic
objects are vertices and edges, and we want to study how they can be
grouped together according to rules that come from an underlying matrix
of relations. This matrix may encode different types of information
depending on the context: adjacency in a static network, transition
probabilities of a Markov chain, similarity scores in data, temporal
adjacency in a dynamic network, or even correlation measures coming from
quantum systems. What matters is that the matrix gives us a way to decide
when a vertex and an edge ``belong together''.

Formally, a \emph{graph multivector} is a finite set of vertices and
edges of a graph that we merge into a single unit. A collection of such
multivectors is called a \emph{graph multivector field} if it partitions
the graph: every vertex and every edge belongs to exactly one
multivector. In this way, the multivector field is a coarse version of
the graph itself, where the basic elements are not single vertices or
edges but larger grouped structures.

The interest of this construction is that once we have such a
partition, we can define a higher--level graph (the $M$--graph which is the graph generated by the multivector field itself looking at the multi valued map we will define later on) whose
nodes are the multivectors and whose edges describe possible transitions
between them. The strongly connected components of this $M$--graph then
form the Morse decomposition of the original system. By varying the
parameters that control the merging process (for example, thresholds
applied to the entries of the relation matrix), we obtain a nested family
of multivector fields. This naturally leads to a persistent version of
Morse decomposition, and hence to persistence diagrams that summarize the
robust recurrent structures present in the system.
\subsection{Generality on multivector fields} 
Classical approaches to dynamics on graphs typically assign states and evolution rules solely to vertices, treating edges as passive carriers of adjacency or transition information. Such vertex-based models are sufficient for path-following dynamics, but they are inadequate for describing coherent local behavior that is intrinsically supported by interactions between vertices. In particular, recurrent or trapping phenomena on graphs often involve both vertices and edges and cannot be localized to single nodes without loss of structural information. To encode such dynamics in a way that remains intrinsic to the graph structure, we adopt a combinatorial approach in which vertices and edges jointly form the underlying state space. Multivector fields are then introduced as locally closed, connected subsets of this space, providing a natural generalization of discrete vector fields that allows for multivalued and locally coherent dynamics while remaining compatible with Conley--Morse theory on finite topological spaces.

\subsubsection{Basic notions}

We begin by recalling the basic language used in the theory of
combinatorial multivector fields and adapting it to the context of graphs.
Throughout this section all sets are finite.

\begin{definition}[Finite poset]
A \emph{finite poset} $(X,\leq)$ is a finite set $X$ equipped with a
reflexive, antisymmetric and transitive relation $\leq$.  Given
$x,y\in X$, we write $x<y$ if $x\leq y$ and $x\neq y$.
For every subset $A\subseteq X$ we define its \emph{closure}
\[
\operatorname{cl}(A) = \{\,x\in X : \exists\,a\in A,\ x\leq a\,\}
\]
and its \emph{mouth}
\[
\operatorname{mo}(A) = \operatorname{cl}(A)\setminus A.
\]
\end{definition}

Every finite poset $(X,\leq)$ defines a topological space, known as the
\emph{Alexandroff} or \emph{$T_0$ topology} (see \cite{Alexandroff_ftop}), in which the closed sets are
precisely the down-sets, that is, subsets $U\subseteq X$ satisfying
$x\in U$ and $y\leq x$ imply $y\in U$.  The open sets are the up-sets,
and the closure operator defined above coincides with the topological
closure in this topology.  In this way, posets and finite $T_0$
topological spaces can be used interchangeably.
\begin{theorem}\label{thm:finite_top} \cite{Alexandroff_ftop}
For a finite poset $(P,\leq)$ the family $\cT_{\leq}$ of upper sets of $\leq$ is a $T_0$ topology on $P$. For a finite $T_0$ topological space $(X,\cT)$ the relation $x \leq_{\cT} y$ defined by $x \in \cl_{\cT} \{ y \}$ is a partial order on $X$. Moreover, the two associations that relate $T_0$ topologies and partial orders are mutually inverse.
    
\end{theorem}

\begin{definition}[Locally closed set]\cite{lipinski2019conley}
A subset $A\subseteq X$ is said to be \emph{locally closed} if it can be
written as the intersection of an open and a closed set.  In the poset
language this is equivalent to the condition that $A$ is
\emph{convex}, i.e.
\[
x,z\in A,\quad x\leq y\leq z \ \Rightarrow\ y\in A.
\]
\end{definition}

Locally closed subsets play a central role because they correspond to
the admissible building blocks in a multivector field.  They ensure that
each multivector interacts with its neighbourhood in a well-defined way,
so that the closure and mouth operations behave consistently with the
underlying topology.

\begin{definition}[Combinatorial multivector field]\cite{lipinski2019conley}
Let $(X,\leq)$ be a finite poset.  A \emph{combinatorial multivector
field} on $X$ is a partition $\mathcal{V}$ of $X$ into finitely many
non-empty, locally closed subsets called \emph{multivectors}.
Elements of the same multivector are considered to belong to a single
local region of the space, and the dynamics is encoded by the relations
between multivectors induced by the order relation of $X$.
\end{definition}

In this setting, one can define an \emph{image} map
$\Pi_{\mathcal{V}}: X\to 2^{X}$ by
\[
\Pi_{\mathcal{V}}(x)= [x]_{\mathcal{V}}\cup \operatorname{cl}([x]_{\mathcal{V}}),
\]
where $[x]_{\mathcal{V}}$ denotes the multivector containing $x$.  The
induced directed graph on the multivectors, with an arrow from $V$ to
$W$ whenever $W\cap\Pi_{\mathcal{V}}(V)\neq\emptyset$, is known as the
\emph{M-graph}.  Its strongly connected components form the Morse
decomposition of the underlying combinatorial system.

A \emph{solution} of a combinatorial dynamical system $\Pi_\cV:X\multimap X$ in $A\subset X$ is a partial map $\varphi:\mathbb{Z}\nrightarrow A$ whose domain, denoted $\dom \varphi$,  is a $\mathbb{Z}$-interval and for any $i,i+1\in \dom \varphi$ the inclusion $\varphi(i+1)\in \Pi_\cV(\varphi(i))$ holds. Let us denote by $\textup{Sol}(A)$ the set of all solutions $\varphi$ such that $\img\varphi\subset A$. $\textup{Sol}(X)$ is the set of all solution of $\Pi_\cV$.
If $\dom\varphi$ is a bounded interval then we say that $\varphi$ is a \emph{path}.
If $\dom\varphi=\ZZ$ then $\varphi$ is a \emph{full solution}. 

A full solution $\varphi : \mathbb{Z} \rightarrow X$ is \emph{left-essential} (respectively \emph{right-essential})
if for every regular $x \in \img\varphi$ the set $\{ t \in \mathbb{Z}\mid \varphi(t) \notin [x]_{\mathcal{V}} \}$ is left-infinite (respectively right-infinite). 
We say that $\varphi$ is \emph{essential} if it is both left- and right-essential.
The collection of all essential solutions $\varphi$ such that $\img\varphi\subset A$ is denoted by $\textup{eSol}(A)$.

The \emph{invariant} part of a set $A\subset X$ is 
$\Inv A := \bigcup \{\img \varphi\mid \varphi\in\esol(A)\}$.
In particular, if $\textup{Inv} A = A$ we say that $A$ is an \emph{invariant set} for a multivector field $\mathcal{V}$.

A closed set $N\subset X$ \emph{isolates} invariant set $S \subset N$ if the following conditions are satisfied:
\begin{enumerate}[label=(\roman*)]
    \item every path in $N$ with endpoints in $S$ is a path in $S$,
    \item $\Pi_{\mathcal{V}}(S) \subset N$.
\end{enumerate}
In this case, $N$ is an \emph{isolating set} for $S$. 
If an invariant set $S$ admits an isolating set then we say that $S$ is an \emph{isolated invariant set}.
The \emph{homological Conley index} of an isolated invariant set $S$ is defined as $\operatorname{Con}(S):=H(\operatorname{cl} S,\mo S)$.

Let $A \subset X$.
By $\bigl\langle A \bigl\rangle_{\mathcal{V}}$ we denote the intersection of 
all locally closed and $\cV$-compatible sets in $X$ containing $A$.
We call this set the $\mathcal{V}$-\emph{hull} of $A$. 
The combinatorial $\alpha$-\emph{limit set} and $\omega$-\emph{limit set} for a full solution $\varphi$ are defined as
\begin{align*}
    & \alpha(\varphi) := \Bigl\langle \bigcap\limits_{t \in \mathbb{Z}^-}\varphi((-\infty,t]) \Bigl\rangle_\mathcal{V}\ , \\
    & \omega (\varphi) := \Bigl\langle \bigcap\limits_{t \in \mathbb{Z}^+}\varphi([t, \infty))  \Bigl\rangle_\mathcal{V}\ .
\end{align*}

\subsubsection{Purpose and Dynamical Interpretation}

The notion of a multivector field was originally designed to bring the
tools of topological dynamics into the discrete and combinatorial
setting.  In a continuous dynamical system, the behaviour of trajectories
is often organized by invariant sets such as equilibria, periodic
orbits, or more complicated recurrent structures.  The central goal is
to describe how these invariant regions are connected and how the
dynamics moves between them.  The same philosophy carries over to the
discrete world once we replace the phase space by a finite poset and the
continuous flow by the combinatorial relation induced by the order.

A \emph{combinatorial multivector field} provides the basic framework for
this description.  Each multivector can be viewed as a coarse
approximation of a local flow cell: a collection of points that evolve
together under the combinatorial dynamics.  The structure of how these
multivectors influence one another, through their closures and mouths,
captures the essential global organization of the system.

For two \emph{combinatorial multivector field} $\mathcal{V}_1$ and $\mathcal{V}_2$, we say that $\mathcal{V}_1\subseteq\mathcal{V}_2$ if:
\begin{align*}
    \forall V \in \mathcal{V}_1, \exists W\in \mathcal{V}_2~|~ V\subseteq W
\end{align*}

\begin{definition}[Morse decomposition]
Let $S\subset X$ be a $\cV$-compatible, invariant set.
Then, a finite collection $\mathcal{M}=\{M_p\subset S\mid p\in\mathbb{P}\}$ is called a \emph{Morse decomposition} of $S$ if there exists a finite poset $(\mathbb{P},\le)$ such that the following conditions are satisfied:
\begin{enumerate}[label=(\roman*)]
    \item $\mathcal{M}$ is a family of mutually disjoint, isolated invariant subsets of $S$,
    \item for every $\varphi\in\esol(S)$ either $\img\varphi \subset M_r$  for an $r \in \mathbb{P}$
or there exist $p, q \in \mathbb{P}$ such that $q > p$,    $\alpha(\varphi)\subset M_q \text{, and } \omega(\varphi)\subset M_p$.
 \end{enumerate}
We refer to the elements of $\mathcal{M}$ as \emph{Morse sets}.
\end{definition}

Let $\mathcal{V}$ be a multivector field on a finite poset $X$ and let
$G_{\mathcal{V}}$ denote the associated $M$--graph, whose vertices are
the multivectors and whose edges are defined by
\[
(V,W)\in E(G_{\mathcal{V}})\quad \Longleftrightarrow \quad
W\cap \Pi_{\mathcal{V}}(V)\neq\emptyset.
\]
The \emph{Morse decomposition} of $\mathcal{V}$ is the collection of
strongly connected components of $G_{\mathcal{V}}$.  Each component
represents a region of recurrent dynamics, meaning that any path within
the $M$--graph that starts in the component can return to it.

The Morse decomposition thus partitions the dynamics into invariant
pieces, ordered by the reachability relation of the $M$--graph.  It
provides a finite and combinatorial description of the recurrent
behaviour of the system, analogous to the decomposition of a continuous
flow into chain-recurrent components.

While the Morse decomposition reveals the structure of recurrence,
finer algebraic invariants can be associated to each component through
homological constructions.  These invariants, called \emph{Conley
indices}, measure the topology of isolated invariant sets and are stable
under perturbations.

\begin{definition}[Conley index \cite{lipinski2019conley}]
Let $M$ be a Morse set of $\mathcal{V}$, i.e.\ a strongly connected
component of $G_{\mathcal{V}}$.  Denote by $\operatorname{cl}(M)$ its
closure in the poset topology and by $\operatorname{mo}(M)$ its mouth.
The \emph{Conley index} of $M$ is defined as the relative homology group
\[
CH_k(M) = H_k\big(\operatorname{cl}(M), \operatorname{mo}(M)\big).
\]
\end{definition}

The Conley index plays a role similar to that of homology groups of
isolating blocks in classical Conley theory.  It detects the type of
recurrent behaviour enclosed by $M$ (for example, whether it corresponds
to a fixed point, a periodic orbit, or a higher-dimensional recurrent
set) and remains invariant under small perturbations of the multivector
field (see \cite{mrozek2022combinatorial}).  In the discrete framework, these indices provide algebraic
signatures of the Morse sets that can be compared, tracked, and
visualized through persistence diagrams.

\medskip

In summary, the purpose of introducing multivector fields is to build a
discrete topological structure that allows us to:
\begin{itemize}
    \item capture local and global recurrence in a purely combinatorial
    form;
    \item define the Morse decomposition as a hierarchy of invariant
    components;
    \item associate algebraic invariants (Conley indices) to these
    components;
    \item and study the stability and persistence of such structures
    under varying parameters.
\end{itemize}
In the next section we transfer these ideas to the graph setting,
where the space $X$ is composed of vertices and edges and the relations
between them are derived from a matrix of weights or probabilities.
\subsection{From Combinatorial to Graph Multivector Fields}

The framework introduced above applies to any finite poset.  In
particular, graphs provide a natural and highly flexible instance of
such a structure.  Since in general graph can also be seen as simplicial complexes, every graph can be regarded as a finite topological
space once we specify an order relation compatible with its incidence
structure.  This correspondence allows us to translate the general
notions of multivectors, closure, mouth, and local closedness into the
language of vertices and edges.

\begin{definition}[Graph poset]
Let $G=(V,E)$ be a finite directed graph.  We define a poset
$(X_G,\leq)$ by setting
\[
X_G = V\cup E,
\]
and declaring that for every edge $e=(v_i,v_j)\in E$ we have
\[
v_i \leq e, \qquad v_j \leq e.
\]
The relation is then extended by reflexivity and transitivity.
Elements of $V$ are called \emph{vertex elements}, and elements of $E$
are called \emph{edge elements}.  The resulting poset encodes the
incidence structure of the graph: vertices precede their incident
edges.
\end{definition}

The $T_0$ topology associated with $(X_G,\leq)$ is generated by these
incidence relations.  In this topology, a set is closed if it contains
every edge together with its incident vertices, and open if it contains
every vertex together with all edges incident to it.  This makes the
notions of closure and mouth intuitive: the closure of a subset includes
its incident elements, while the mouth records the immediate boundary of
interaction between different parts of the graph.

\begin{definition}[Local closedness in graphs]
A subset $A\subseteq X_G$ is said to be \emph{locally closed} if it is
convex with respect to the order relation of the graph poset, that is,
whenever $x,z\in A$ and $x\leq y\leq z$ in $X_G$, we have $y\in A$.  In
combinatorial terms, $A$ is locally closed if and only if it satisfies
the following:
\begin{itemize}
    \item if an edge $(v_i,v_j)\in A$, then its incident vertices
    $v_i,v_j$ also belong to $A$;
    \item if a vertex $v_i\in A$ and an edge $e=(v_i,v_j)$ shares this
    vertex with another vertex $v_j\in A$, then the edge $e$ must also
    belong to $A$.
\end{itemize}
\end{definition}

This local closedness condition ensures that every subset representing a
coherent piece of the graph is topologically consistent: if an element
belongs to it, all elements that are incident through the graph relation
are also contained, up to the chosen level of adjacency.

\subsection{Graph Multivector Fields}

We can now specialize the notion of a multivector field to the graph
setting.

\begin{definition}[Graph multivector field]
Let $G=(V,E)$ be a finite directed graph and $(X_G,\leq)$ its associated
graph poset.  A \emph{graph multivector field} or simply \emph{multivector field} on $G$ is a partition
\[
\mathcal{V}_G = \{V_1, V_2, \dots, V_m\}
\]
of $X_G$ into non-empty, locally closed subsets called
\emph{graph multivectors} or just \emph{multivectord}.  Each $V_i$ may contain both vertices and
edges, and the family $\mathcal{V}_G$ satisfies:
\begin{enumerate}
    \item[(i)] every vertex and every edge of $G$ belongs to exactly one
    multivector;
    \item[(ii)] if two multivectors intersect, they are merged into a
    single multivector (so the family forms a true partition);
    \item[(iii)] the local closedness condition holds for each
    multivector, ensuring compatibility with the graph topology.
\end{enumerate}
\end{definition}

Each multivector represents a local coherent unit of the graph---a
collection of vertices and edges that interact strongly according to the
chosen relation or weighting rule.  The multivector field
$\mathcal{V}_G$ thus provides a coarse description of the graph in terms
of higher-order units.

Given a graph multivector field $\mathcal{V}_G$, one defines the
associated \emph{$M$--graph} in the same way as in the combinatorial
setting (see \cite{lipinski2019conley}): the vertices of the $M$--graph correspond to the multivectors,
and there is a directed edge from $V_i$ to $V_j$ whenever
$V_j \cap \Pi_{\mathcal{V}_G}(V_i) \neq \emptyset$, where
$\Pi_{\mathcal{V}_G}(x) = [x]_{\mathcal{V}_G} \cup
\operatorname{cl}([x]_{\mathcal{V}_G})$.  The strongly connected
components of this $M$--graph form the Morse decomposition of the
original graph structure, and their corresponding Conley indices capture
the algebraic and topological features of the recurrent regions. Here we will focus on invariant sets as qualitative descriptors of global dynamics, rather
than on individual solutions.
\medskip

This construction extends the classical theory of multivector fields (see \cite{dey2022tracking,Woukeng_2024,dey2019persistent}) to
graphs in a way that preserves its topological meaning while enabling
new applications.  It provides a natural discrete framework to study
complex relational systems such as networks, probabilistic processes,
and temporal or quantum graphs under a unified topological formalism.
\subsection{Basic propositions}

The next collection of propositions gathers elementary properties
of graph multivector fields that are frequently used in constructions
and proofs.  We state them here for convenience and reference.

\begin{proposition}[Existence -- trivial multivector fields]
\label{prop:existence}
Let \(G=(V,E)\) be a finite directed graph and let \(X_G=V\cup E\) be
its graph poset.  Then the following are graph multivector fields:
\begin{enumerate}
  \item the \emph{singleton partition} \(\mathcal{V}_{\mathrm{sing}}=\{\{x\} : x\in X_G\}\);
  \item the \emph{whole-space partition} \(\mathcal{V}_{\mathrm{all}}=\{X_G\}\).
\end{enumerate}
Hence graph multivector fields always exist.
\end{proposition}

\begin{proof}
Every singleton \(\{x\}\) is locally closed (it is convex in a finite
poset), and singletons form a partition of \(X_G\); thus
\(\mathcal{V}_{\mathrm{sing}}\) satisfies the definition.  The whole
space \(X_G\) is also locally closed (trivially convex) and the single
block \(\{X_G\}\) is a partition, so \(\mathcal{V}_{\mathrm{all}}\) is
also a graph multivector field.  This proves existence.
\end{proof}

\begin{proposition} 
\label{prop:coarsening-lc-fixed}
Let $\mathcal V=\{V_i\}_{i\in I}$ be a multivector field on
$X=V\cup E$ whose blocks are locally closed (with respect to the face-order
Alexandroff topology). Let $J\subset I$ and set
\[
U \;:=\; \bigcup_{j\in J} V_j.
\]
Assume that $U$ is itself locally closed, i.e. $\operatorname{LocalClosure}(U)=U$.
Form the family
\[
\mathcal V' \;:=\; \bigl(\mathcal V\setminus\{V_j:j\in J\}\bigr)\;\cup\;\{U\},
\]
obtained by replacing the blocks $\{V_j:j\in J\}$ by their union $U$.
Then $\mathcal V'$ is a multivector field whose blocks are locally closed.
\end{proposition}

\begin{proof}
Since $\mathcal V$ is a partition of $X$ and $\{V_j:j\in J\}$ is a subcollection
of pairwise disjoint blocks, the set
\[
\mathcal V\setminus\{V_j:j\in J\}
\]
is a family of pairwise disjoint sets that together with $U$ still cover $X$.
Indeed, every element of $X$ either lies in one of the removed $V_j$ (hence
lies in $U$) or lies in some block of $\mathcal V$ not in the removed
subcollection (and is therefore still covered). Thus $\mathcal V'$ is a
partition of $X$.

It remains to check that every block of $\mathcal V'$ is locally closed.
By hypothesis every block $V_i$ of $\mathcal V$ is locally closed. All
blocks of $\mathcal V'$ coming from $\mathcal V\setminus\{V_j:j\in J\}$
retain this property unchanged. The only new (or modified) block is $U$;
by assumption $\operatorname{LocalClosure}(U)=U$, so $U$ is locally closed.
Therefore every block of $\mathcal V'$ is locally closed.

Consequently $\mathcal V'$ is a partition of $X$ into locally-closed blocks,
i.e. a multivector field, as required.
\end{proof}
\begin{proposition}[Containment of Morse sets under coarsening]
\label{prop:morse-coarsening}
Let $\mathcal V_{1}$ and $\mathcal V_{2}$ be two graph multivector fields 
such that $\mathcal V_{1} \subseteq \mathcal V_{2}$. 
Then for every Morse set $M_1 \in \mathcal M(\mathcal V_{1})$, 
there exists a Morse set $M_2 \in \mathcal M(\mathcal V_{2})$ such that
\[
M_1 \subseteq M_2.
\]
Consequently, each Morse set in $\mathcal M(\mathcal V_{2})$ is a union of Morse sets 
in $\mathcal M(\mathcal V_{1})$, and the number of Morse sets cannot increase.
\end{proposition}

\begin{proof}
Assume that $\mathcal V_{1} \subseteq \mathcal V_{2}$, 
so that every multivector in $\mathcal V_{1}$ is contained in a multivector of 
$\mathcal V_{2}$.  
This inclusion induces an inclusion of the associated $M$--graphs:
\[
G_{1} \subseteq G_{2}.
\]
Each Morse set corresponds to a strongly connected component (SCC) of its $M$--graph. 
Since $G_{2}$ contains all edges of $G_{1}$ and possibly additional ones, 
any directed path in $G_{1}$ remains a directed path in $G_{2}$. 
When new edges are added that connect two previously disjoint SCCs and form a directed cycle 
between them, these SCCs merge into a single strongly connected component.  
Thus, each SCC of $G_{1}$ is contained in some SCC of $G_{2}$, 
and every Morse set at the coarser level is a union of Morse sets at the finer level.  
As merging cannot produce new components, the total number of Morse sets decreases or remains constant.
\end{proof}

\begin{remark}
This proposition formalizes the intuitive fact that new Morse sets created
after coarsening are never genuinely new entities: each arises as a union
of previous recurrent components.  Hence the family of Morse
decompositions evolves monotonically across any filtration of coarsenings.
This property underlies the stability of persistence diagrams
(see~Edelsbrunner--Harer~\cite[Thm.~7.3]{EdelsbrunnerHarer2010CT} and
Chazal--de~Silva--Oudot~\cite[Thm.~3.1]{ChazalDeSilvaOudot2016Stability}).
\end{remark}

\medskip
\noindent
The propositions above illustrate that graph multivector fields form a
flexible yet well-structured framework: they always exist, can be
refined or coarsened by merging locally closed parts, and their
associated $M$--graphs give rise to finite Morse decompositions that
summarize the recurrent organization of the system.  What remains is to
explain how such fields are \emph{constructed} from data.  In practice,
graph multivector fields arise from weighted or probabilistic relations
that quantify the interaction strength between vertices and edges.  The
construction process determines which elements of the graph should be
merged into a single multivector and how the structure evolves as a
function of a chosen threshold.

\section{Construction of Graph Multivector Fields}
\label{sec:construction}

The aim of this section is to formalize the way in which graph
multivector fields are obtained from numerical or logical data.  In most
applications, a graph $G=(V,E)$ is endowed with an additional matrix of
relations that encodes the strength or probability of interactions
between its elements.  These relations can represent transition
probabilities in a Markov process, similarity scores in a data set,
weights in a neural or transportation network, or coupling coefficients
in a quantum or temporal graph.  The goal is to translate these
quantitative relations into a combinatorial structure---the graph
multivector field---that preserves the topological coherence of the
graph while capturing the strength of connections.

\subsection{Relation matrices}

\begin{definition}[Inclusion between graphs]
\label{def:graph-inclusion}
Let $G_{1}=(V,E_{1})$ and
$G_{2}=(V,E_{2})$ be two directed graphs constructed from
the same vertex set $V$.
We say that $G_{1}$ is \emph{included} in $G_{2}$,
and write
\[
G_{1} \subseteq G_{2},
\]
whenever $E_{1}\subseteq E_{2}$.
Inclusion thus means that every edge appearing at level $E_1$ also
appears at level $E_2$.
\end{definition}
 
\subsubsection*{Thresholded relation matrices}

Let \(W = [W(i,j)]_{1\le i,j\le n}\) be a relation matrix on the
finite vertex set \(V=\{v_1,\dots,v_n\}\).  We do \emph{not} assume any
symmetry or stochastic property of \(W\); its entries are arbitrary real
numbers (typically normalized to \([0,1]\) in applications).

For a real parameter \(\lambda\) define the {\bf thresholded relation matrix}
\(w_\lambda\) by the entrywise rule
\begin{equation}\label{eq:threshold-matrix}
    w_\lambda(v_i,v_j) \;=\; 
    \begin{cases}
      1, & \text{if } W(i,j) > \lambda,\\[4pt]
      0, & \text{otherwise.}
    \end{cases}
\end{equation}
Equivalently, \(w_\lambda = \mathbf{1}_{\{W>\lambda\}}\) as an entrywise
indicator.

We then view \(w_\lambda\) as the (binary) relation used by the merging
algorithm: a (directed) trigger for merging the vertex \(v_i\) with the
edge representing \(\{v_i,v_j\}\) is the condition \(w_\lambda(v_i,v_j)=1\).
(Recall our convention that each unordered pair \(\{v_i,v_j\}\) is represented
by a single edge \((v_i,v_j)\) with \(i<j\); both asymmetric comparisons
\(w_\lambda(v_i,v_j)\) and \(w_\lambda(v_j,v_i)\) may independently trigger
merging into that same edge element.)

\medskip

\begin{enumerate}
  \item \textbf{Monotonicity in \(\lambda\).} For fixed indices \(i,j\) the map
    \(\lambda \mapsto w_\lambda(v_i,v_j)\) is nonincreasing: if
    \(\lambda_1 < \lambda_2\) then
    \[
      w_{\lambda_1}(v_i,v_j) \ge w_{\lambda_2}(v_i,v_j).
    \]
    In words: lowering the threshold \(\lambda\) can only create additional
    ones (make more relations active), raising \(\lambda\) can only turn
    ones into zeros.

  \item \textbf{Induced graphs and inclusion direction.}  Let \(G_\lambda\)
    be the graph induced by \(w_\lambda\) via the edge set
    \(\{(v_i,v_j)\mid i<j,\ w_\lambda(v_i,v_j)=1 \text{ or } w_\lambda(v_j,v_i)=1\}\).
    Because \(w_\lambda\) is nonincreasing in \(\lambda\), the family
    \((G_\lambda)_\lambda\) is nested \emph{in the opposite} direction:
    if \(\lambda_1 < \lambda_2\) then
    \[
      G_{\lambda_1} \supseteq G_{\lambda_2}.
    \]
    (Lower thresholds include more active relations.)
    We can decide to keep the notation where increasing the parameter
    produces coarser partitions, simply reparameterize by
    \(\mu=-\lambda\) or reverse the order in all statements; either choice is
    fine as long as the convention is declared clearly.

\end{enumerate}
The thresholding will be used in this paper to dynamically monitor the progression of the relations encoded in our relation matrix. This can be for example relations relations between states in a Markov chain.

\subsection{Merging rules}
\label{sec:merging-rules}

Given a finite graph $G=(V,E)$ with
\[
E=\{(v_i,v_j)\mid i<j,\ v_i,v_j\in V\},
\]
and an original relation matrix $W=[W(i,j)]_{i,j=1}^n$,
we define for each threshold $\lambda\in\mathbb{R}$ the \emph{thresholded relation matrix}
$w_\lambda=[w_\lambda(v_i,v_j)]_{i,j=1}^n$ by
\begin{equation}\label{eq:wlambda}
w_\lambda(v_i,v_j)=
\begin{cases}
1,& \text{if } W(i,j)>\lambda,\\[3pt]
0,& \text{otherwise.}
\end{cases}
\end{equation}
No symmetry is assumed for $W$, so $w_\lambda(v_i,v_j)$ and
$w_\lambda(v_j,v_i)$ may differ.
The matrix $w_\lambda$ encodes which relations are considered active at
threshold $\lambda$ and drives the merging process in the construction of
the graph multivector field~$\mathcal V_\lambda$.

\medskip
The combinatorial space on which $\mathcal V_\lambda$ is defined is
\[
X := V \cup E,
\]
which contains both the vertices and the edges of the graph.
The set $E$ is fixed for all thresholds, while the active relations in
$w_\lambda$ determine how the elements of $X$ are merged.

\medskip
Starting from singleton vertices and singleton edges, we iteratively apply the
following deterministic rules to obtain a well-defined partition
$\mathcal V_\lambda$:

\begin{itemize}
    \item \textbf{Inclusion of all elements.}
    Every vertex $v_i\in V$ and every edge $e=(v_i,v_j)\in E$
    appears in $\mathcal V_\lambda$.  
    Even if no relation involving an edge is active, that edge remains as a
    singleton multivector.

    \item \textbf{Relation-based merging.}
    Merging occurs whenever a relation is active, that is, when
    $w_\lambda(v_i,v_j)=1$ or $w_\lambda(v_j,v_i)=1$.  Specifically:
    \begin{itemize}
        \item If $w_\lambda(v_i,v_j)=1$, merge the vertex $v_i$ with the edge
              $(v_i,v_j)$.
        \item If $w_\lambda(v_j,v_i)=1$, merge the vertex $v_j$ with the edge
              $(v_i,v_j)$.
        \item If both relations are active, merge $v_i$, $v_j$, and the edge
              $(v_i,v_j)$ into a single multivector.
    \end{itemize}

    \item \textbf{Transitive merging.}
    If two multivectors share an element, they are merged into a single one.
    This operation is applied transitively: if
    $P_1\cap P_2\neq\varnothing$ and $P_2\cap P_3\neq\varnothing$,
    then $P_1$, $P_2$, and $P_3$ are all merged together.
    The transitive rule guarantees that $\mathcal V_\lambda$ forms a true
    partition of $X$.

    \item \textbf{Singleton retention.}
    Any vertex or edge not involved in a merge remains as a singleton
    multivector in $\mathcal V_\lambda$.
    This ensures that the partition always covers $X$ completely.

    \item \textbf{Monotonicity with respect to $\lambda$.}
    The function $\lambda \mapsto w_\lambda$ is nonincreasing:
    decreasing the threshold activates more relations.
    Consequently, if $\lambda_1<\lambda_2$, then
    \[
        \mathcal V_{\lambda_1} \succeq \mathcal V_{\lambda_2}.
    \] 

    \item \textbf{Critical values.} The matrix \(W\) has only finitely many
    distinct entries.  The multivector field \(\mathcal V_\lambda\) (and the
    induced $M$–graph) can change only when \(\lambda\) crosses a value in
    the finite set \(\{\,W(i,j): 1\le i,j\le n\,\}\).  Thus there are finitely
    many distinct \(\mathcal V_\lambda\) as \(\lambda\) varies.

  \item \textbf{Directionality / single-edge representation.} Although each
    unordered pair \(\{v_i,v_j\}\) is represented by a single edge element
    \((v_i,v_j)\) with \(i<j\), the two comparisons \(W(i,j)>\lambda\) and
    \(W(j,i)>\lambda\) are evaluated independently.  Either (or both) can
    trigger merges that include the same edge element.  This preserves
    asymmetric information in the merging process while keeping the edge set
    non-redundant.
\end{itemize}

\medskip
Together these rules ensure that:
\begin{enumerate}
    \item $\mathcal V_\lambda$ is a partition of $X=V\cup E$ into
          locally closed, edge-closed subsets;
    \item the family $(\mathcal V_\lambda)_\lambda$ is finite and evolves
          monotonically by coarsening as $\lambda$ increases;
    \item all edges and vertices are always represented, even when inactive.
\end{enumerate}

\begin{remark}
\label{rem:edges_singletons-final}
Edges corresponding to weak or inactive relations, i.e.\ those for which
both $W(i,j)\le\lambda$ and $W(j,i)\le\lambda$, remain singleton
multivectors throughout the filtration.  
These isolated edges often represent negligible or transient interactions
and correspond to trivial Morse sets in the $M$--graph representation.
\end{remark}
\begin{remark}
    We should note that the reason here we construct multivector fields is that they can encode locally
coherent interaction patterns that cannot be captured by single vertices or edges, capturing more information.
\end{remark}

The rules above define the foundation of the construction algorithm
presented in the next subsection.

We will now describe the construction of the graph multivector field
$\mathcal V_\lambda$ associated with a relation matrix $W_\lambda$.
The setting is a finite directed graph
$G=(V,E)$ in which all possible edges between distinct vertices are present.
The edge set $E$ does not depend on $\lambda$; the threshold $\lambda$
controls only the relations that are strong enough to trigger merging in
the construction of the multivector field.

\medskip
\noindent
The underlying combinatorial space on which $\mathcal V_\lambda$ is defined is
\[
X := V \cup E,
\]
that is, the union of all vertices and all directed edges.
Each element of $X$ is considered as an atomic object, and edges are distinct
from their endpoints.  At the beginning of the process, every vertex and every
edge is represented by a singleton multivector.  These singletons are then
merged following the rules of Section~\ref{sec:merging-rules} according to the
values of $W_\lambda$.

\begin{remark}
\label{rem:initial_partition}
Because $E$ is fixed across thresholds, every vertex and edge of $G$ appears
in every multivector field $\mathcal V_\lambda$.  If an edge never satisfies a
merging condition at threshold $\lambda$, it remains a singleton multivector.
This ensures that $\mathcal V_\lambda$ always partitions the entire space $X$
and that the family $(\mathcal V_\lambda)_\lambda$ evolves by coarsening,
never by deletion.
\end{remark}

\subsection{Algorithmic construction of the graph multivector field}
\label{sec:algorithm}
In this subsection, we start with a relation matrix, obtained from a graph $G=(V,E)$, with $V=\{v_1,v_2,v_3,...,v_n\}$, the set of vertices, and the goal is to be able to construct a \emph{ graph combinatorial multivector field}, a tool that will be central in our study.

We now describe the algorithm that constructs the graph multivector field
$\mathcal V_\lambda$ associated with a general relation matrix
$W_\lambda = [w_\lambda(v_i,v_j)]_{i,j=1}^n$, without any symmetry or
normalization assumption.  
The construction begins with all vertices and all (undirected) edges of the
graph as singleton multivectors, then merges them progressively according to
the relation values compared to the threshold~$\lambda$.

The edge set is defined as
\[
E = \{ (v_i,v_j) \mid i < j,\ v_i,v_j \in V \},
\]
so that each pair of vertices is represented by a single edge, avoiding
repetition of $(v_i,v_j)$ and $(v_j,v_i)$.  
The resulting multivector field $\mathcal V_\lambda$ is a partition of
$X = V \cup E$ that evolves monotonically with $\lambda$.

\begin{algorithm}
\caption{Construction of the Graph Multivector Field from a Relation Matrix}
\label{alg:graph-multivector-field}
\begin{algorithmic}[1]
\Require Relation matrix \( W = [W(i,j)] \), threshold \( \lambda \)
\Ensure Graph multivector field \( \mathcal{V}_\lambda \)

\State Initialize \( \mathcal{V}_\lambda \gets \emptyset \) 
      \Comment{Start with an empty family of multivectors}
\State Define \( V = \{v_1, v_2, \dots, v_n\} \) as the vertex set
\State Define \( E = \{(v_i,v_j) : i < j,\ v_i,v_j \in V\} \) 
      \Comment{Avoid duplicated edges}

\Comment{Step 1: Threshold the relation matrix}
\ForAll{\( i,j \) with \( i \neq j \)}
    \State \( w_\lambda(v_i,v_j) \gets 
            \begin{cases}
            1, & \text{if } W(i,j) > \lambda,\\
            0, & \text{otherwise.}
            \end{cases} \)
\EndFor

\Comment{Step 2: Initialize all vertices and edges as singleton multivectors}
\ForAll{ \( v_i \in V \) }
    \State \( \mathcal{V}_\lambda \gets \mathcal{V}_\lambda \cup \{\{v_i\}\} \)
\EndFor
\ForAll{ \( (v_i,v_j) \in E \) }
    \State \( \mathcal{V}_\lambda \gets \mathcal{V}_\lambda \cup \{\{(v_i,v_j)\}\} \)
\EndFor

\Comment{Step 3: Merge vertices with edges based on active relations}
\ForAll{ \( (v_i,v_j) \in E \) }
    \If{ \( w_\lambda(v_i,v_j) = 1 \) }
        \State Merge \( v_i \) with edge \( (v_i,v_j) \)
        \State \( \mathcal{V}_\lambda \gets 
            \mathcal{V}_\lambda \setminus \{\{v_i\}, \{(v_i,v_j)\}\} 
            \cup \{\{v_i,(v_i,v_j)\}\} \)
    \EndIf
    \If{ \( w_\lambda(v_j,v_i) = 1 \) }
        \State Merge \( v_j \) with edge \( (v_i,v_j) \)
        \State \( \mathcal{V}_\lambda \gets 
            \mathcal{V}_\lambda \setminus \{\{v_j\}, \{(v_i,v_j)\}\} 
            \cup \{\{v_j,(v_i,v_j)\}\} \)
    \EndIf
\EndFor

\Comment{Step 4: Merge overlapping multivectors transitively}
\ForAll{ pairs \( A,B \in \mathcal{V}_\lambda \) }
    \If{ \( A \cap B \neq \emptyset \) }
        \State Merge \( A \) and \( B \): 
            \( \mathcal{V}_\lambda \gets 
              \mathcal{V}_\lambda \setminus \{A,B\} 
              \cup \{A \cup B\} \)
    \EndIf
\EndFor

\State \Return \( \mathcal{V}_\lambda \) 
      \Comment{Return the multivector field at threshold \( \lambda \)}
\end{algorithmic}
\end{algorithm}

\begin{remark}
\label{rem:directionality-final}
Although $E$ contains only one copy of each edge $(v_i,v_j)$ with $i<j$,
 Algorithm~\ref{alg:graph-multivector-field} still accounts for the asymmetry of the matrix $W_\lambda$:
if $w_\lambda(v_i,v_j)>\lambda$ or $w_\lambda(v_j,v_i)>\lambda$, then
the corresponding vertices and the edge $(v_i,v_j)$ are merged.  
This allows directional effects to influence the merging process without
duplicating edges.
\end{remark}

\begin{remark}
\label{rem:singleton_edges}
If an edge $(v_i,v_j)\in E$ never satisfies any merging condition,
it remains as a singleton multivector in $\mathcal V_\lambda$.
These singleton edges still participate in the $M$--graph and may
form trivial Morse sets.  This guarantees that no element of the
graph is ever discarded, preserving a complete combinatorial
representation of the system for all thresholds.
\end{remark}

\medskip

\subsection*{Structural properties (local-closedness version)}
\label{sec:properties-locallyclosed}
\begin{proposition}[Multivectors are connected]
\label{rem:multivectors-connected}
Every block produced by Algorithm~\ref{alg:graph-multivector-field} is a
connected subset of the incidence graph on \(X=V\cup E\).

\begin{proof}
The algorithm starts from singleton blocks (vertices and edges), which are
trivially connected. Each merge step combines two blocks whenever there is
an active relation between elements in the two blocks (a vertex-edge trigger
or the symmetric check). Merging two connected sets that share an edge or a
vertex yields a connected union. The transitive merging step further
unites any intersecting blocks; unions of blocks that intersect are
connected because intersection provides the linking element. Hence every
block produced at every stage is connected, and so is the final block set.
\end{proof}
\end{proposition}

\begin{proposition}[Connected subsets are locally closed]
\label{prop:connected-implies-locallyclosed}
Let \(X=V\cup E\) be the incidence set of a finite graph \(G=(V,E)\) equipped
with the face order (so \(v\le e\) iff \(v\) is an endpoint of \(e\)) and
the associated upper Alexandroff topology. If \(A\subset X\) is connected
(in the induced incidence graph), then \(A\) is locally closed.
\end{proposition}

\begin{proof}
We work with the face order and the corresponding upper Alexandroff closure:
for any set \(S\subset X\),
\[
\operatorname{cl}(S)=\{\,y\in X : y\le x \text{ for some } x\in S\,\}.
\]
We can see that the only nontrivial order relations are vertex \(\le\) edge
relations: if \(e=(v_i,v_j)\in E\) then \(v_i\le e\) and \(v_j\le e\),
and there are no relations \(e\le v\) or non-incidence relations.

Let \(A\subset X\) be connected. We compute \(\operatorname{cl}(A)\).
By the closure formula above, any element added to \(A\) by closure must
be a face of some element of \(A\). Because only vertices are faces of
edges, the closure of \(A\) is obtained by adjoining to \(A\) those
vertices that are endpoints of edges in \(A\). Concretely define
\[
V_A := \{\, v\in V : \exists\, e\in A\cap E \text{ with } v\text{ an endpoint of }e \,\}.
\]
Then
\[
\operatorname{cl}(A)=A\ \cup\ V_A.
\]
(Vertices in \(A\) are already in the union; vertices not incident to any
edge of \(A\) do not belong to \(V_A\).)

Therefore the mouth of \(A\) is precisely
\[
\operatorname{mo}(A)=\operatorname{cl}(A)\setminus A = V_A\setminus (A\cap V).
\]
In particular \(\operatorname{mo}(A)\subseteq V\), i.e.\ the mouth is a
subset of the vertex set.

In the upper Alexandroff topology any singleton vertex \(\{v\}\) is closed,
because there are no elements strictly below a vertex in the face order.
Consequently any subset of \(V\) is closed (finite union of closed singletons).
Hence \(\operatorname{mo}(A)\) is closed.

Since the mouth then \(A\) is locally
closed by definition (local-closedness = mouth is closed). This completes the
proof.
\end{proof}

\begin{remark} 
\label{rem:connected-union-locallyclosed}
In the face-order Alexandroff topology, any connected subset of
$X=V\cup E$ is locally closed
(Proposition~\ref{prop:connected-implies-locallyclosed}).
Therefore, when the union
$U=\bigcup_{j\in J}V_j$ of blocks in a multivector field is connected,
the hypothesis $\operatorname{LocalClosure}(U)=U$ in
Proposition~\ref{prop:coarsening-lc-fixed} is automatically satisfied.
In this frequent case, checking local closedness is unnecessary:
the connectedness of $U$ already guarantees it.
\end{remark}

\begin{proposition} 
\label{prop:termination-lc-final}
For any relation matrix $W$ from a finite graph \(G=(V,E)\), and any threshold \(\lambda\), 
Algorithm~\ref{alg:graph-multivector-field} terminates after finitely many
operations and outputs for the fixed threshold $\lambda$ a graph multivector field
\(\mathcal V_\lambda\), that is a partition of \(X = V \cup E\) into
nonempty, connected, and hence locally closed blocks.
\end{proposition}

\begin{proof}
We verify termination and correctness directly from the steps of
Algorithm~\ref{alg:graph-multivector-field}.

\paragraph{(1) Finiteness and termination.}
The algorithm operates on the finite set
\(X = V \cup E\), where \(|V|=n\) and \(|E| = \frac{n(n-1)}{2}\).
Initialization (Steps~1–2) creates the family of singletons
\(\{\{x\}:x\in X\}\).
Each subsequent merge operation (Steps~3–4) replaces at least two distinct
blocks by their union.  Hence every successful merge strictly decreases the
number of blocks in \(\mathcal V_\lambda\).  Because the number of blocks is
bounded below by \(1\), the process can perform at most \(|X|-1\) successful
merges and therefore terminates in finitely many steps.

\paragraph{(2) Partition property.}
At every stage of the algorithm, the following invariants hold:
\begin{itemize}
    \item all blocks are nonempty subsets of \(X\);
    \item the union of all blocks equals \(X\);
    \item distinct blocks are disjoint.
\end{itemize}
These properties are clear at initialization.  
Step~3 (relation-based merges) replaces two disjoint singletons by their
union, which preserves both coverage and disjointness.  
Step~4 (transitive merging) explicitly checks for nonempty intersections and
replaces intersecting pairs by their union, ensuring that in the next
iteration no overlaps remain.  
When the algorithm stops, no further merges are triggered, so all blocks are
pairwise disjoint and their union remains \(X\).  Thus
\(\mathcal V_\lambda\) is a partition of \(X\) into nonempty blocks.

\paragraph{(3) Connectedness and local closedness.}
By construction, the algorithm starts with singleton blocks, which are
connected in the incidence graph on \(X\).  
Each merge in Step~3 or Step~4 joins two blocks that share a vertex or an
edge; this preserves connectedness because the intersection element forms a
path between them.  
Inductively, every intermediate block and every final block is connected.
By Proposition~\ref{prop:connected-implies-locallyclosed},
every connected subset of \(X\) is locally closed in the face-order
Alexandroff topology.  Therefore all blocks in the final family
\(\mathcal V_\lambda\) are locally closed.

\paragraph{(4) Conclusion.}
Algorithm~\ref{alg:graph-multivector-field} thus terminates after finitely
many merges and produces a family \(\mathcal V_\lambda\) that is a partition
of \(X\) into nonempty, connected, and locally closed blocks.
\end{proof}

\begin{proposition}[Filtration and monotonicity]
\label{prop:filtration-monotonic}
Let \(W=[W(i,j)]\) be the original relation matrix and for each threshold
\(\lambda\in\mathbb{R}\) define the binary matrix
\(w_\lambda(i,j)=\mathbf{1}_{\{W(i,j)>\lambda\}}\).
Let \(\mathcal V_\lambda\) be the graph multivector field produced by
Algorithm~\ref{alg:graph-multivector-field}. Then:

\begin{enumerate}
\item If \(\lambda_1 < \lambda_2\), then
      \[
      \mathcal V_{\lambda_1} \succeq \mathcal V_{\lambda_2},
      \]
      i.e. the partition \(\mathcal V_{\lambda_1}\) is a coarsening of
      \(\mathcal V_{\lambda_2}\).
      Equivalently, as \(\lambda\) decreases, new merges may occur but
      existing merges are never undone.

\item The associated $M$--graphs satisfy a monotonicity relation:
      when \(\lambda\) decreases, every adjacency relation present at
      \(\lambda_2\) remains present at \(\lambda_1<\lambda_2\),
      possibly together with new ones.  
      Consequently, the strongly connected components (Morse sets)
      can only grow or merge; no new disconnected components appear.
\end{enumerate}
\end{proposition}

\begin{proof}
(\emph{1. Monotonicity of merges.})  
By definition \(w_\lambda(i,j)=1\) if and only if \(W(i,j)>\lambda\).
Hence if \(\lambda_1<\lambda_2\), then
\(w_{\lambda_1}(i,j)\ge w_{\lambda_2}(i,j)\) for all \(i,j\).
That is, every active relation at level \(\lambda_2\) remains active at
level \(\lambda_1\), while additional pairs may become active.
Since merges in Algorithm~\ref{alg:graph-multivector-field} are triggered
whenever a relation is active, all merges performed at \(\lambda_2\)
also occur at \(\lambda_1\).
Thus the partition at the smaller threshold \(\lambda_1\) is obtained
by further merging blocks of \(\mathcal V_{\lambda_2}\), and we have
\(\mathcal V_{\lambda_1} \succeq \mathcal V_{\lambda_2}\).

(\emph{2. Monotonicity of $M$--graphs and Morse sets.})  
The $M$--graph \(G_{\mathcal V_\lambda}\) encodes adjacency between
multivectors determined by mouth and closure relations.
Because no adjacency relation is removed when $\lambda$ decreases—
only new ones may be added—the edge set of
\(G_{\mathcal V_{\lambda_2}}\) is contained in that of
\(G_{\mathcal V_{\lambda_1}}\) for \(\lambda_1<\lambda_2\).
Consequently, the strongly connected components (SCCs) of
\(G_{\mathcal V_{\lambda_1}}\) are unions of SCCs of
\(G_{\mathcal V_{\lambda_2}}\); each coarse SCC at the smaller threshold
contains one or more fine SCCs from the larger threshold.
Therefore, the number of Morse sets cannot increase as \(\lambda\)
decreases, and every new Morse set at lower thresholds is formed by
merging earlier ones.
\end{proof}

\begin{remark}
\label{prop:critical-lc}
The multivector field \(\mathcal V_\lambda\)  
can change only when \(\lambda\) crosses a value in the finite set
\(\Lambda^*=\{W(i,j):i\ne j\}\). Hence there are finitely many distinct multivector fields in the filtration.
\end{remark}

\begin{proposition}[Homology above dimension one vanishes]
\label{prop:homology-lc}
For every Morse set \(M\) obtained from a graph multivector field,
\[
CH_k(M)=H_k(\operatorname{cl}(M),\operatorname{mo}(M))=0\quad\text{for }k>1.
\]
\end{proposition}

\begin{proof}
\(\operatorname{cl}(M)\) is contained in the one-dimensional cell complex
given by the underlying graph; relative homology for graphs vanishes in
dimensions \(>1\), hence the Conley index is trivial above dimension one.
\end{proof}

\begin{remark}
\label{rem:complexity-lc}
Computing active relations \(W(i,j)>\lambda\) requires \(O(n^2)\)
comparisons. The transitive closure / union step can be implemented using
a union–find (disjoint-set) structure over the universe \(X\), giving near
linear performance in the number of successful merges.
\end{remark}

\medskip
The propositions above establish that the construction of
graph multivector fields yields a finite, nested, and well-behaved
filtration indexed by the threshold parameter~$\lambda$.
Each level of the filtration defines a locally closed partition of the
combinatorial space \( X = V \cup E \) whose induced
$M$--graph captures the dynamic connectivity structure of the system.
To illustrate these concepts concretely, we now present a simple
example demonstrating the full construction process:
from the relation matrix and its thresholding,
to the resulting multivector field at a fixed threshold.

\begin{example}[Construction at a fixed threshold]
\label{ex:construction}
Consider a directed graph \(G=(V,E)\) with vertex set
\(V=\{v_1,v_2,v_3\}\) and relation matrix
\[
W =
\begin{pmatrix}
0 & 0.8 & 0.3 \\
0.5 & 0 & 0.7 \\
0.2 & 0.6 & 0
\end{pmatrix}.
\]
Each entry \(W(i,j)\) measures the strength of relation from \(v_i\) to \(v_j\).
The edge set of \(G\) is
\[
E = \{(v_1,v_2), (v_1,v_3), (v_2,v_3)\},
\]
so the combinatorial space is
\[
X = V \cup E = \{v_1, v_2, v_3, (v_1,v_2), (v_1,v_3), (v_2,v_3)\}.
\]

\medskip\noindent
\textbf{Step 1. Thresholding.}
Fix the threshold \(\lambda = 0.6\).
The binary thresholded relation matrix is then
\[
w_\lambda =
\begin{pmatrix}
0 & 1 & 0 \\
0 & 0 & 1 \\
0 & 0 & 0
\end{pmatrix},
\]
meaning that \(w_\lambda(v_i,v_j)=1\) when \(W(i,j)>\lambda\).

\medskip\noindent
\textbf{Step 2. Initialization.}
We begin with singleton multivectors for every vertex and every edge:
\[
\mathcal V_\lambda^{(0)} =
\{\{v_1\}, \{v_2\}, \{v_3\},
  \{(v_1,v_2)\}, \{(v_1,v_3)\}, \{(v_2,v_3)\}\}.
\]

\medskip\noindent
\textbf{Step 3. Merging according to active relations.}
From the nonzero entries of \(w_\lambda\):
\begin{itemize}
  \item \(w_\lambda(v_1,v_2)=1\): merge \(v_1\) and \((v_1,v_2)\);
  \item \(w_\lambda(v_2,v_3)=1\): merge \(v_2\) and \((v_2,v_3)\).
\end{itemize}
After these merges we have
\[
\mathcal V_\lambda^{(1)} =
\{\{v_1,(v_1,v_2)\}, \{v_2,(v_2,v_3)\}, \{v_3\}, \{(v_1,v_3)\}\}.
\]

\medskip\noindent
\textbf{Step 4. Transitive merging.}
No blocks share common elements, so no further merges occur.
Thus the final multivector field at threshold \(\lambda=0.6\) is
\[
\mathcal V_\lambda = 
\{A_1=\{v_1,(v_1,v_2)\}, A_2=\{v_2,(v_2,v_3)\}, A_3=\{v_3\}, A_4=\{(v_1,v_3)\}\}.
\]

\medskip\noindent
\textbf{Step 5. Interpretation.}
Each block of \(\mathcal V_\lambda\) is connected in the incidence graph
and hence locally closed.
The $M$--graph corresponding to this multivector field has four nodes,
one per block, with directed edges induced by adjacency of their
constituent vertices and edges in \(G\).
As \(\lambda\) decreases, more entries of \(w_\lambda\) become active,
leading to additional merges and a coarser multivector field,
illustrating the nested filtration
\(\mathcal V_{\lambda_1} \succeq \mathcal V_{\lambda_2}\)
for \(\lambda_1<\lambda_2\).

\paragraph{Step 6: Morse decomposition and Conley indices.}

From the locally closed multivector field obtained at threshold
$\lambda=0.6$,
\[
A_1=\{v_1,e_{12}\}, \qquad
A_2=\{v_2,e_{23}\}, \qquad
A_3=\{v_3\}, \qquad
A_4=\{e_{13}\},
\]
we compute the mouths using the incidence (face) relation:
\[
\begin{aligned}
\operatorname{mo}(A_1)&=\{v_2\}, &
\operatorname{mo}(A_2)&=\{v_3\}, \\
\operatorname{mo}(A_3)&=\varnothing, &
\operatorname{mo}(A_4)&=\{v_1,v_3\}.
\end{aligned}
\]
The Conley-Morse--graph is then constructed according to the rule
\[
A_i \longrightarrow A_j
\quad\Longleftrightarrow\quad
\operatorname{mo}(A_i)\cap A_j\neq\varnothing.
\]
This yields the directed edges
\[
A_1\to A_2, \qquad
A_2\to A_3, \qquad
A_4\to A_1, \qquad
A_4\to A_3.
\]
No other intersections occur.  The resulting Conley-Morse--graph is shown in
Figure~\ref{fig:Mgraph-example}.

\medskip\noindent
\textbf{Strongly connected components.}
A strongly connected component (SCC) is a maximal subset of nodes
such that each pair of nodes is mutually reachable by directed paths.
Applying Tarjan’s algorithm to the above $M$--graph yields four singleton
SCCs:
\[
\mathcal{M}_1=\{A_1\},\qquad
\mathcal{M}_2=\{A_2\},\qquad
\mathcal{M}_3=\{A_3\},\qquad
\mathcal{M}_4=\{A_4\}.
\]
Hence the Morse decomposition at $\lambda=0.6$ consists of the four
individual blocks of the multivector field.

\medskip\noindent
\textbf{Conley index dimensions.}
For each Morse set $M$, the Conley index is computed as
$CH_k(M)=H_k(\operatorname{cl}(M),\operatorname{mo}(M))$.
Since all closures lie in the one--dimensional graph complex,
homology above dimension~1 vanishes.
A direct computation gives
\[
\begin{array}{c|cc}
\text{Morse set } M & \dim CH_0(M) & \dim CH_1(M) \\
\hline
A_1 & 0 & 0\\
A_2 & 0 & 0\\
A_3 & 1 & 0\\
A_4 & 0 & 1
\end{array}
\]
Thus $A_3$ corresponds to an isolated fixed vertex
(nontrivial zero-dimensional index), while $A_4$ contributes a
nontrivial one-dimensional Conley index, reflecting the presence of a
directed loop through the edge $(v_1,v_3)$.
The sets $A_1$ and $A_2$ have trivial indices, indicating transient
connections in the dynamics.

\begin{figure}[ht]
\centering
\begin{tikzpicture}[>=stealth, node distance=3.0cm, every node/.style={font=\small}]
  \node[draw, circle, thick] (A1) at (0,1.5) {$A_1(0,0)$};
  \node[draw, circle, thick] (A2) at (0,-1.5) {$A_2(0,0)$};
  \node[draw, circle, thick] (A4) at (3,0) {$A_4(0,1)$};
  \node[draw, circle, thick] (A3) at (6,0) {$A_3(1,0)$};

  \draw[->, thick] (A1) to[bend right=10] (A2); 
  \draw[->, thick] (A2) to[bend right=10] (A3);

  \draw[->, thick] (A4) to[bend right=15] (A1);
  \draw[->, thick] (A4) to[bend left=10] (A2); 
    

\end{tikzpicture}
\caption{Conley-Morse-graph for $\lambda=0.6$.
Edges are drawn whenever $\operatorname{mo}(A_i)\cap A_j\neq\varnothing$.
Each node is annotated with the pair
$(\dim CH_0,\dim CH_1)$}
\label{fig:Mgraph-example}
\end{figure}
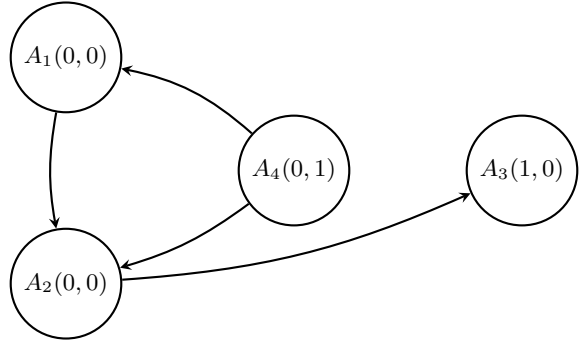

\end{example}
 
Up to this point, we have established how a graph multivector field
captures the combinatorial structure of a system at a fixed threshold
and how its Morse decomposition and Conley indices describe the
underlying dynamical organization in terms of invariant components.
However, a single threshold provides only a static snapshot.
As the threshold $\lambda$ varies, the multivector field evolves:
Morse sets may appear, merge, or disappear, and their Conley indices
may change accordingly.
To understand the system in its entirety, it is therefore essential
to study how these Morse decompositions transform across the filtration
$\{\,\mathcal{V}_\lambda\,\}_\lambda$.
This leads naturally to the notion of \emph{persistence}, where we follow
the evolution of the invariant sets and their indices across scales,
obtaining a global description of the system’s multi-level organization.

\section{Morse decomposition, Conley indices and persistent structures}
\label{sec:morse-persistence}

The construction of graph multivector fields provides a discrete,
combinatorial framework for describing the internal organization of a
relation matrix or weighted graph.
At each threshold $\lambda$, the algorithm yields a partition
$\mathcal{V}_\lambda$ of the space $X=V\cup E$ into connected,
locally closed blocks whose interactions define the $M$--graph.
The strongly connected components of this $M$--graph constitute the
\emph{Morse decomposition} of the system at level $\lambda$, and their
Conley indices capture the algebraic structure of the corresponding
invariant sets.
The Morse decomposition organizes the structure of the $M$-graph into finitely many
invariant components, called Morse sets. Each Morse set corresponds to a maximal
substructure of the $M$-graph in which recurrence occurs. The admissible transitions
between these sets induce a partial order, reflecting how invariant components
are connected through the directed structure of the $M$-graph. In this way,
the Morse decomposition provides a finite combinatorial summary of the global
organization encoded by the graph multivector field.

While each fixed threshold offers a consistent dynamical description,
it represents only one scale of observation.  
As $\lambda$ varies, the relation matrix $W_\lambda$ changes
monotonically: edges are progressively removed as the threshold increases,
and consequently, multivectors may merge or vanish.  
This induces a nested family of graph multivector fields
\[
  \mathcal{V}_{\lambda_1} \succeq \mathcal{V}_{\lambda_2}
  \quad \text{for } \lambda_1 < \lambda_2,
\]
and therefore a corresponding filtration of $M$--graphs and Morse
decompositions.  
Studying how the Morse sets and their Conley indices evolve along this
filtration provides a way to understand the global organization of the
system across scales.

This section introduces the notion of \emph{persistence} in the context of
graph multivector fields.  
We interpret persistence as the continuous tracking of Morse sets,
together with their Conley indices, throughout the filtration
$\{\mathcal{V}_\lambda\}_\lambda$.  
This approach allows us to identify stable invariant structures that
persist across a range of thresholds, as well as transient or unstable
features that appear only within narrow intervals.  
In this way, persistence provides a unified perspective on the evolution of
combinatorial dynamics and their topological signatures, extending the
classical ideas of persistent homology to the setting of multivector
fields and Morse decompositions.

 \subsection{Morse decomposition across the filtration}
\label{sec:morse-filtration}

For each threshold $\lambda$, the graph multivector field
$\mathcal V_\lambda$ induces an $M$--graph
$G_{\mathcal V_\lambda}$ whose strongly connected components (SCCs)
constitute the \emph{Morse decomposition} of the combinatorial dynamics
at that scale.  
We denote these Morse sets by
\[
\mathcal M(\lambda)
  = \{ M_1(\lambda), \dots, M_{r(\lambda)}(\lambda) \},
\]
where each $M_i(\lambda)$ is a strongly connected component of
$G_{\mathcal V_\lambda}$.  
Each Morse set represents a recurrent region of the system, that is, a
subset of vertices and edges that can mutually reach one another through
the relations encoded by $\mathcal V_\lambda$.

\medskip
\noindent\textbf{Evolution along the filtration.}
Under the convention
\[
w_\lambda(i,j)=\mathbf{1}_{\{W(i,j)>\lambda\}},
\]
lowering the threshold $\lambda$ activates additional relations.
Hence, for $\lambda_1 < \lambda_2$, the boolean matrices satisfy
$w_{\lambda_1} \ge w_{\lambda_2}$ entrywise, so more merges can occur at
$\lambda_1$.  
The corresponding multivector fields satisfy the monotonicity relation
\[
\mathcal V_{\lambda_1} \succeq \mathcal V_{\lambda_2}
\qquad\text{for all } \lambda_1<\lambda_2,
\]
meaning that the partition of $X=V\cup E$ becomes \emph{coarser} as the
threshold decreases.  
Consequently, the Morse decomposition can only coarsen: existing Morse
sets may merge into larger ones as $\lambda$ decreases, but no new Morse
sets can be created.

To describe their evolution, we scan the filtration from large to small
thresholds (from fine to coarse resolution).  
A Morse set $M$ is said to be \emph{born} at the first (largest)
threshold $\lambda_b$ where it appears as a strongly connected component
in $G_{\mathcal V_{\lambda_b}}$, and it \emph{dies} at the first
(smaller) threshold $\lambda_d < \lambda_b$ where it merges into a
larger component.  
The persistence (lifetime) of $M$ is measured by the interval
$(\lambda_d,\lambda_b]$, or equivalently by its length
$\lambda_b - \lambda_d$.  
If $M$ persists until the minimal threshold of the filtration, we set
$\lambda_d$ equal to that minimal value. Formal definitions of birth and death in our setting is given here.
\begin{definition}\label{birth-death}
    Let $\{\mathcal{M}_\lambda\}_{\lambda \in \Lambda}$ be a parameterized family
of Morse decomposition of our $M$-graphs indexed by a totally ordered parameter set $\Lambda$.
A \emph{birth} occurs at parameter value $\lambda$ if a new Morse set
appears in the Morse decomposition of $\mathcal{M}_\lambda$ that was not
present for any smaller parameter value.
A \emph{death} occurs at $\lambda$ if a Morse set present for smaller
parameter values ceases to exist as an isolated invariant component,
for example due to merging with another component or loss of isolation.
Birth and death therefore correspond to structural changes in the
decomposition of the $M$-graph into invariant components.
\end{definition}
\begin{remark}
    When we said in Definition \ref{birth-death} that a new Morse set appears, in general we just mean that there is a merging between two  or many different Morse sets, with a change of the structure of the resulting Morse set ( for example a change in the Conley index of the resulting Morse set).
\end{remark}

\begin{definition}[Morse persistence diagram]
The \emph{Morse persistence diagram} of a graph $G$ is the multiset of
points $(\lambda_b,\lambda_d )$
corresponding to all Morse sets appearing across the filtration
$\{\,G_{\mathcal V_\lambda}\,\}_\lambda$.  
Each point encodes the birth and death thresholds of one recurrent
structure and quantifies its persistence across scales.
\end{definition}

In practice, the diagram can be obtained by tracking how strongly
connected components of the $M$--graph merge as $\lambda$ decreases.
Points far from the diagonal (long-lived features) correspond to robust
recurrent structures, whereas points near the diagonal represent noisy interactions.

\medskip
\noindent
\textbf{Conley indices along the filtration.}
For each Morse set $M(\lambda)$ we associate its Conley index
\[
CH_k(M(\lambda)) =
H_k(\operatorname{cl}(M(\lambda)), \operatorname{mo}(M(\lambda))).
\]
The rank of $CH_k(M(\lambda))$ encodes the local topological structure
of the recurrent region.  
For graphs, only degrees $k=0$ and $k=1$ can be nontrivial: the
degree--$0$ index detects connected components, while the degree--$1$
index captures cycles or feedback loops within the Morse set.  
Tracking these indices along the filtration enriches the Morse
persistence diagram with topological information.

\begin{definition}[Conley--enriched Morse persistence diagram]
For each $k\in\{0,1\}$, the \emph{Conley--enriched Morse persistence
diagram} is the multiset of tuples
\[
(\lambda_b, \lambda_d, k, \beta_k),
\]
where $(\lambda_b,\lambda_d)$ records the birth and death thresholds of
a Morse set, and $\beta_k=\operatorname{rank}\,CH_k(M(\lambda_b))$ is
the $k$--th Betti number of its Conley index at birth.
\end{definition}

\begin{proposition}[Finiteness of the Morse persistence diagram]
\label{prop:finite-diagram}
Let $\{\mathcal V_\lambda\}$ be the filtration of graph multivector
fields obtained from a finite relation matrix $W$.  Then the Morse
persistence diagram of $G$ contains finitely many points.  
Every birth or death occurs at one of the finitely many critical
thresholds
\[
\Lambda^*=\{\,W(i,j):\,i\ne j\,\}.
\]
\end{proposition}

\begin{proof}
The thresholding rule
$w_\lambda(i,j)=\mathbf{1}_{\{W(i,j)>\lambda\}}$
can change only when $\lambda$ crosses one of the finitely many
values $W(i,j)$.
Between two consecutive thresholds
$\lambda^{(t)} < \lambda^{(t+1)}$ of $\Lambda^*$,
the boolean matrix $w_\lambda$ and the resulting multivector field
$\mathcal V_\lambda$ remain constant.  
Hence the family $\{\mathcal V_\lambda\}_\lambda$ takes only finitely
many distinct values.  
Since the $M$--graph associated to each $\mathcal V_\lambda$ is finite,
its collection of strongly connected components—and therefore the set of
Morse sets across all levels—is finite as well.  
Each birth or death event must coincide with a change in
$\mathcal V_\lambda$, so all such events occur at parameters belonging to
$\Lambda^*$.  
Thus the Morse persistence diagram contains finitely many points.
\end{proof}

For each fixed $M$-graph, the Conley index is an invariant of its Morse sets. As the parameter varies and the structure of the $M$-graph
changes, Morse sets may merge, split, or cease to be isolated.
In such cases, the associated Conley indices may change, reflecting a genuine
modification of the invariant structure encoded by the graph multivector field.
Thus, variations of the Conley index across parameters record structural
changes in the combinatorial dynamics.

Persistence is defined at the level of Morse sets of the parameterized
family of $M$-graphs. Each Morse set represents a maximal invariant
component in the directed structure. When defined, the Conley index
of a Morse set provides an additional invariant attached to that component.
Thus, persistence tracks the appearance, disappearance, and evolution
of Morse sets across parameters, together with the associated Conley
indices when these are defined.

\medskip
The Morse persistence diagram offers a compact, interpretable summary of
how recurrent structures evolve through the filtration.

\subsection{Stability of Morse persistence diagrams}
\label{sec:stability} 

The purpose of this section is to establish that the Morse persistence
diagrams obtained from the graph multivector field construction are
\emph{stable} under small perturbations of the input relation matrix.
Stability, originally proved for classical persistence modules by
Cohen--Steiner, Edelsbrunner, and Harer~\cite{cohen2005stability}, and
later generalized by Chazal, De~Silva, Glisse, and
Oudot~\cite{chazal2016structure}, ensures that small perturbations in the
data result in small changes in the persistence diagram.
We adapt these arguments to the combinatorial dynamics generated by
graph multivector fields, relying on the finite topological framework of
Mrozek~\cite{mrozek2017conley} and the homological setting described in
\cite{kaczynski2004computational}.
 
In the previous subsection, we defined the Morse persistence diagram as
a record of how the recurrent structures (Morse sets) of the graph evolve
through the filtration induced by the thresholds $\lambda$.
We now ask: if the relation matrix $W$ describing the system is slightly
perturbed—for example due to noise or uncertainty—how does the diagram
change?
Intuitively, we expect the diagram to deform continuously, with all birth
and death coordinates shifting by at most the perturbation amplitude.
The goal of this section is to formalize this expectation.

\subsubsection{Preliminaries and metrics}

We begin by introducing the metrics used to compare relation matrices and
persistence diagrams.

\begin{definition}[Sup--norm between relation matrices]
Given two relation matrices $W=(W(i,j))$ and $W'=(W'(i,j))$ on the same
finite set $X_G$, their \emph{sup--norm distance} is defined by
\[
\|W-W'\|_\infty = \sup_{i,j\in X_G} |W(i,j)-W'(i,j)|.
\]
This norm measures the largest possible change in relation strength
between any two graph elements and serves as the perturbation amplitude.
It is the standard choice in persistence stability analysis
\cite{cohen2005stability,chazal2016structure}.
\end{definition}

\begin{definition}[Bottleneck distance]
Let $D_1$ and $D_2$ be two persistence diagrams, i.e.\ multisets of
points $(b,d)\in\{(x,y)\in\mathbb R^2:x<y\}\cup\{(x,\infty)\}$.
The \emph{bottleneck distance} between $D_1$ and $D_2$ is
\[
d_B(D_1,D_2)
  = \inf_{\eta}
     \sup_{x\in D_1}
     \|x-\eta(x)\|_\infty,
\]
where $\eta$ ranges over all bijections between $D_1$ and
$D_2\cup\Delta$, and $\Delta=\{(t,t)\}$ denotes the diagonal.
This metric measures the largest necessary shift to match births and
deaths between the two diagrams.
\end{definition}

\subsubsection{Perturbation of the filtration}

We now analyze how small perturbations in the matrix $W$ affect the
filtration $\{\mathcal V_\lambda(W)\}_\lambda$ generated by the
multivector field algorithm (Algorithm~\ref{alg:graph-multivector-field}).
Recall that the algorithm depends only on the boolean thresholded matrix
\[
w_\lambda(i,j) = \mathbf{1}_{\{W(i,j) > \lambda\}},
\]
so changes in $W$ of size at most $\varepsilon$ can only alter the
threshold activation of relations within $\varepsilon$ of the boundary.\begin{lemma}[Filtration stability]
\label{lem:filtration_stability}
Let $W$ and $W'$ be two relation matrices on the same finite set
with $\|W - W'\|_\infty \le \varepsilon$.  
Then, for every threshold $\lambda \in [0,1]$, the corresponding
multivector field partitions satisfy
\[
\mathcal V_{\lambda+\varepsilon}(W)
  \preceq
\mathcal V_{\lambda}(W')
  \preceq
\mathcal V_{\lambda-\varepsilon}(W),
\]
where $\mathcal A \preceq \mathcal B$ means that
$\mathcal A$ is a \emph{refinement} of $\mathcal B$
(every block of $\mathcal A$ is contained in some block of $\mathcal B$).
Equivalently, as $\lambda$ decreases, the partition becomes coarser.
\end{lemma}

\begin{proof}
Recall that the algorithm depends only on the thresholded relation
matrix
\[
w_\lambda(i,j) = \mathbf{1}_{\{W(i,j) > \lambda\}},
\]
and that lowering $\lambda$ activates additional entries (more edges
become ``active''), producing coarser partitions.

Now fix $\lambda$ and assume $\|W - W'\|_\infty \le \varepsilon$.
This means that for every pair $(i,j)$,
\[
|W(i,j) - W'(i,j)| \le \varepsilon.
\]

\textbf{Step 1.}  
Suppose $W(i,j) > \lambda + \varepsilon$.  
Then automatically $W'(i,j) \ge W(i,j) - \varepsilon > \lambda$.  
Hence every pair $(i,j)$ that is ``active'' at level $\lambda + \varepsilon$
for $W$ remains active at level $\lambda$ for $W'$.  
Therefore, every merge that occurs in $\mathcal V_{\lambda+\varepsilon}(W)$
must also occur in $\mathcal V_{\lambda}(W')$.
This implies
\[
\mathcal V_{\lambda+\varepsilon}(W) \preceq \mathcal V_{\lambda}(W').
\]

\textbf{Step 2.}  
Conversely, if $W'(i,j) > \lambda$, then
$W(i,j) \ge W'(i,j) - \varepsilon > \lambda - \varepsilon$,
so all relations active in $\mathcal V_{\lambda}(W')$ are also active in
$\mathcal V_{\lambda - \varepsilon}(W)$.
Thus the merges realized in $\mathcal V_{\lambda}(W')$ also appear in
$\mathcal V_{\lambda - \varepsilon}(W)$, giving
\[
\mathcal V_{\lambda}(W') \preceq \mathcal V_{\lambda - \varepsilon}(W).
\]

Combining the two inclusions, we obtain
\[
\mathcal V_{\lambda+\varepsilon}(W)
  \preceq
\mathcal V_{\lambda}(W')
  \preceq
\mathcal V_{\lambda-\varepsilon}(W).
\]
Since lowering $\lambda$ always coarsens the partition, this result
captures the precise way in which perturbations of $W$ translate into
a small horizontal shift of the entire filtration by at most $\varepsilon$.
\end{proof}

\begin{lemma} 
\label{lem:morse_refinement}
Under the assumptions of Lemma~\ref{lem:filtration_stability}, 
the families of Morse decompositions 
$\{\mathcal M_W(\lambda)\}$ and $\{\mathcal M_{W'}(\lambda)\}$ 
satisfy the inclusion relations
\[
\mathcal M_W(\lambda+\varepsilon)
  \prec
\mathcal M_{W'}(\lambda)
  \prec
\mathcal M_W(\lambda-\varepsilon)
  \qquad\text{for all }\lambda,
\]
where $\mathcal A \prec \mathcal B$ means that 
each Morse set of $\mathcal A$ is contained in a Morse set of $\mathcal B$.
\end{lemma}

\begin{proof}
By Lemma~\ref{lem:filtration_stability}, 
the underlying multivector field partitions satisfy
\[
\mathcal V_{\lambda+\varepsilon}(W)
  \preceq
\mathcal V_{\lambda}(W')
  \preceq
\mathcal V_{\lambda-\varepsilon}(W),
\]
that is, the partition for $W'$ at threshold $\lambda$ lies between those
for $W$ at $\lambda+\varepsilon$ and $\lambda-\varepsilon$.

Each Morse decomposition is obtained by taking the strongly connected 
components (SCCs) of the corresponding $M$--graph. 
When a partition becomes coarser (for smaller $\lambda$), 
the $M$--graph gains additional edges, 
so SCCs can only merge and never split.

Hence, a Morse set appearing at threshold $\lambda$ for $W'$ 
must be contained within a Morse set for $W$ at $\lambda-\varepsilon$, 
and it must contain Morse sets for $W$ at $\lambda+\varepsilon$. 
In other words,
\[
\mathcal M_W(\lambda+\varepsilon)
  \prec
\mathcal M_{W'}(\lambda)
  \prec
\mathcal M_W(\lambda-\varepsilon),
\]
as claimed.

Intuitively, this means that a small perturbation of the relation matrix 
shifts the formation and merging of Morse sets by at most $\varepsilon$ 
along the filtration, but does not change their overall structure.
\end{proof}

\subsubsection{Stability of the Morse persistence diagram}
\label{sec:stability-morse}

One of the most fundamental achievements of persistence theory is its
\emph{stability}: small perturbations of the data lead to only small
changes in the resulting persistence diagram.
This property guarantees that persistent features are robust with respect
to noise, numerical errors, or model uncertainty.
In our framework, perturbations arise as small variations of the relation
matrix $W$, which encodes the pairwise strengths of interaction or
transition between elements of the graph.
The following theorem establishes that the Morse persistence diagram
remains stable under such perturbations.

\begin{theorem}[Stability of Morse persistence diagrams]
\label{thm:diagram_stability}
Let $W$ and $W'$ be two relation matrices on the same finite graph $G$
such that $\|W - W'\|_\infty \le \varepsilon$.
Let $D_W$ and $D_{W'}$ denote the Morse persistence diagrams computed from
the filtrations $\{\mathcal V_\lambda(W)\}$ and
$\{\mathcal V_\lambda(W')\}$, respectively.
Then the diagrams satisfy the bound
\[
d_B(D_W, D_{W'}) \le \varepsilon.
\]
\end{theorem}

\begin{proof}
From Lemma~\ref{lem:morse_refinement}, the Morse decompositions associated
with $W$ and $W'$ satisfy
\[
\mathcal M_W(\lambda+\varepsilon)
  \prec
\mathcal M_{W'}(\lambda)
  \prec
\mathcal M_W(\lambda-\varepsilon)
  \qquad \text{for all }\lambda.
\]
This means that each Morse set in $\mathcal M_{W'}(\lambda)$ is contained
between Morse sets of $\mathcal M_W$ taken at thresholds shifted by at
most $\varepsilon$.

Consider now the birth and death thresholds $(\lambda_b,\lambda_d)$ of a
Morse set in the filtration associated with $W$.
By the inclusion above, when we slightly perturb the relation matrix to
$W'$, the same set (or its image under merging) must appear and disappear
within thresholds shifted by at most $\varepsilon$.
That is, there exist $\lambda_b',\lambda_d'$ for the corresponding Morse
set in $W'$ such that
\[
|\lambda_b - \lambda_b'| \le \varepsilon,
\qquad
|\lambda_d - \lambda_d'| \le \varepsilon.
\]
This defines a one-to-one correspondence between points
$(\lambda_b,\lambda_d) \in D_W$ and
$(\lambda_b',\lambda_d') \in D_{W'}$
such that
\[
\|(\lambda_b,\lambda_d) - (\lambda_b',\lambda_d')\|_\infty \le \varepsilon.
\]
By the definition of the bottleneck distance, this implies
$d_B(D_W, D_{W'}) \le \varepsilon$.

The argument follows the same logic as the classical stability theorem
for ordinary persistence diagrams
(see~\cite{cohen2005stability, chazal2016structure}), 
but is adapted here to the setting of multivector field filtrations and
Morse decompositions.
\end{proof}

\medskip
\noindent 
Theorem~\ref{thm:diagram_stability} establishes that a small change $\varepsilon$ in any entry of $W$---for
instance due to noise in the observed transition probabilities---can
shift the birth and death of recurrent structures by at most $\varepsilon$
along the filtration.
Consequently, long-lived Morse sets correspond to genuinely robust
dynamical features of the system rather than artifacts of numerical
thresholding.
This result provides the theoretical foundation for using Morse
persistence diagrams as stable and interpretable descriptors of dynamic
or relational structures derived from data.
\begin{corollary}[Conley--enriched stability]
\label{cor:conley_stability}
Let $W$ and $W'$ be two relation matrices on the same finite graph with
$\|W - W'\|_\infty \le \varepsilon$.
Assume that the Conley indices $CH_k(M(\lambda))$ are computed over a
fixed field.
Then the Conley--enriched Morse persistence diagrams
(including Betti-number labels $\beta_k$) satisfy
\[
d_B(D_W^{(k)}, D_{W'}^{(k)}) \le \varepsilon.
\]
\end{corollary}

\begin{proof}
The Conley index $CH_k(M(\lambda))$ of a Morse set depends only on the
topological pair $(\operatorname{cl}(M(\lambda)), \operatorname{mo}(M(\lambda)))$,
that is, on its closure and its mouth.
As shown in Lemma~\ref{lem:filtration_stability}, a perturbation of the
relation matrix by at most $\varepsilon$ can only shift the appearance or
disappearance of multivectors---and therefore the boundaries of Morse
sets—by thresholds within $\varepsilon$.
Since both the closure and the mouth vary continuously with the
filtration, the associated homology groups can change only when a merge
or split occurs, i.e.\ when $\lambda$ crosses a critical value displaced
by at most $\varepsilon$.

Because the Betti numbers $\beta_k = \operatorname{rank} CH_k(M(\lambda))$
are integer invariants, they remain constant between such events and
change only at those shifted thresholds.
Consequently, each labelled point
$(\lambda_b,\lambda_d,k,\beta_k)$ in $D_W^{(k)}$ corresponds to a point
$(\lambda_b',\lambda_d',k,\beta_k')$ in $D_{W'}^{(k)}$ satisfying
\[
\|(\lambda_b,\lambda_d) - (\lambda_b',\lambda_d')\|_\infty \le \varepsilon.
\]
Hence the bottleneck distance between the labelled diagrams is bounded by
$\varepsilon$, which proves the claim.
A similar argument for the continuity of Conley indices in combinatorial
settings can be found in~\cite{kaczynski2004computational,
mrozek2017conley}.
\end{proof}

\medskip
\noindent
Corollary~\ref{cor:conley_stability} extends the stability of Morse
persistence to the topological level: not only the existence and duration
of recurrent regions are robust, but also their internal topological
structure as captured by the Conley index.
In particular, the Betti numbers $\beta_0$ and $\beta_1$, which describe
the number of connected components and independent cycles within a Morse
set, remain stable under small perturbations of the relation matrix.
This result reinforces the interpretation of the
Conley--enriched Morse persistence diagram as a reliable descriptor of
dynamic organization: both the recurrence patterns and their topological
signatures vary smoothly with the data.
Such robustness is crucial in applications to noisy or stochastic
systems, where the relation strengths $W(i,j)$ are often estimated from
empirical observations or simulations.
 
Although the present work is purely structural, the Morse persistence
diagram provides a compact and stable representation of the dynamic
organization encoded by the relation matrix. In particular, the diagram
can be embedded into finite-dimensional feature spaces using functional
coordinates such as the polynomial summaries introduced by
Carlsson~\cite{adcock2013ring}. These \emph{Carlsson coordinates}
map a persistence diagram to numerical invariants obtained by integrating
monomials of birth and death parameters over all points of the diagram.
Because the Morse persistence diagram is stable with respect to
perturbations of the relation matrix, these coordinates inherit the same
robustness. Consequently, one may regard the diagram as a structured
descriptor of the underlying graph dynamics, suitable for statistical
comparison, clustering, or model inference, without requiring direct
access to the original relation matrix. In this sense, the present
framework provides a principled bridge between combinatorial dynamics and
data-driven learning.

 \section{Future work}

The present framework establishes the theoretical foundation for
analyzing recurrent structures in graphs through multivector fields and
Morse persistence diagrams.  Several directions naturally extend from
this work.

A first avenue concerns the practical computation of these diagrams on
large or time-varying graphs.  Efficient algorithms for the multivector
construction, together with scalable implementations of Morse set
detection, would enable the systematic study of complex relational data
in realistic settings such as social or biological networks.

A second direction lies in integrating the persistence information with
machine learning models.  Persistence diagrams derived from graph
multivector fields provide stable, low-dimensional representations of
dynamical and relational structure.  They can therefore serve as
features for learning tasks, allowing neural networks or other
statistical models to incorporate topological information about
recurrence, feedback, or entanglement.  Training on such features may
reveal new ways to detect patterns that are invisible to purely metric
methods.

Finally, the same formalism may be adapted to domains where relational
data evolve over time, in logic-based reasoning, or in quantum
correlation networks.  Future work will focus on implementing and
evaluating these applications, exploring how the topological summaries
obtained from Morse persistence can inform, guide, or even constrain
data-driven models in these diverse contexts.
\section*{Competing interests}
   Research of D.W. is partially supported by the Polish National Science Center under Opus Grant No. 2019/35/B/ST1/00874.\\
    
\newpage
\bibliographystyle{plain}
\bibliography{references}
\end{document}